\documentclass[a4paper,11pt]{article}

\usepackage[utf8]{inputenc} 
\usepackage{enumerate}

\usepackage[T1]{fontenc}

\usepackage{hyperref}       
\usepackage{url}            
\usepackage{booktabs}       
\usepackage{amsfonts,amsmath,amssymb,amsthm}       
\usepackage{nicefrac}       
\usepackage{microtype}      
\usepackage{lipsum}
\usepackage{xcolor}
\usepackage{graphicx,float}
\usepackage{tikz}
\usepackage[square,comma,compress,numbers]{natbib}
\usepackage[big]{layaureo}
\usepackage{perpage}
\MakePerPage{footnote}

\theoremstyle{plain}
\newtheorem{theorem}{Theorem}[section]
\newtheorem{lemma}{Lemma}[section]

\newtheorem{thmout}{Theorem}

\theoremstyle{definition}
\theoremstyle{definition}

\newtheorem*{remark}{Remark}

\definecolor{myred}{RGB}{226,56,18}
\definecolor{myorange}{RGB}{228,139,0}
\definecolor{mygreen}{RGB}{4,215,17}
\definecolor{mygrey}{RGB}{180,180,180}

\def\Cx{\mathbb{C}}

\def\dist{\mathrm{dist}}
\def\D{\mathbb{D}}

\begin{document}

\title{\textbf{\textsc{A note on forward iteration of inner functions}}}

\author{Gustavo R.~Ferreira\thanks{Email: \texttt{gustavo.rodrigues-ferreira@open.ac.uk}}\\
  \small{School of Mathematics and Statistics, The Open University}\\
  \small{Milton Keynes, MK7 6AA, UK}
}

\maketitle

\begin{abstract}
A well-known problem in holomorphic dynamics is to obtain Denjoy--Wolff-type results for compositions of self-maps of the unit disc. Here, we tackle the particular case of inner functions: if $f_n:\D\to\D$ are inner functions fixing the origin, we show that a limit function of $f_n\circ\cdots\circ f_1$ is either constant or an inner function. For the special case of Blaschke products, we prove a similar result and show, furthermore, that imposing certain conditions on the speed of convergence guarantees $L^1$ convergence of the boundary extensions. We give a counterexample showing that, without these extra conditions, the boundary extensions may diverge at all points of $\partial\D$.
\end{abstract}


\section{Introduction}
Let $f:\mathbb{D}\to\mathbb{D}$ be a holomorphic function. When taking iterates of $f$, the Denjoy--Wolff theorem -- one of the earliest theorems in complex dynamics -- tells us that, if $f$ is not an elliptic M\"obius transformation, the sequence $(f^n)_{n\in\mathbb{N}}$ converges locally uniformly to a unique constant $c\in\overline{\mathbb{D}}$. A variation of this problem concerns the ``iteration'' of sequences of holomorphic functions $f_n:\mathbb{D}\to\mathbb{D}$, of which there are two flavours:
\[ \text{$F_n := f_n\circ f_{n-1}\circ\cdots\circ f_1$ and $G_n := f_1\circ f_2\circ\cdots\circ f_n$.} \]
The procedure on the left is called \textit{forward iteration}, while the other is called \textit{backward iteration}; sometimes, the names \textit{left} and \textit{right compositions} (respectively) are used. In view of the Denjoy-Wolff theorem, many criteria, pioneered by Lorentzen \cite{Lor90} and Gill \cite{Gil88}, have been found to establish when the limits of forward and backward iteration are constant -- see, for instance, \cite{BCM04}, \cite{CS21}, \cite{AC22}, and \cite[Chapters 11 to 13]{KL07}. For example, Keen and Lakic proved in \cite{KL06} the remarkable result that for any $f:\mathbb{D}\to K$ where $K\subset\D$ is a non-Bloch domain there exists a sequence $f_n:\mathbb{D}\to K$ such that the backward iteration of $(f_n)_{n\in\mathbb{N}}$ converges locally uniformly to $f$. The non-Bloch condition is subtle and geometric in nature, and we will not discuss it further here; see \cite{BCM04} for a definition.

Regarding forward iteration, Benini \textit{et al.} recently gave a more explicit criterion for deciding if the limit function of $f_n\circ \cdots\circ f_1$ is constant  -- under the extra assumption that $f_n(0) = 0$ for all $n\in\mathbb{N}$. More specifically, they showed \cite[Theorem A]{BEFRS19}:
\begin{thmout} \label{thmout:class}
Let $f_n:\D\to\D$ be holomorphic and such that $f_n(0) = 0$ for all $n\in\mathbb{N}$, and assume that the forward iteration of $f_n$ converges locally uniformly to a limit $f:\D\to\D$. Then, $f$ is non-constant if and only if
\begin{equation*}
    \sum_{n\geq 1} (1 - |f_n'(0)|) < +\infty.
\end{equation*}
\end{thmout}

Benini \textit{et al.}'s result was motivated by the study of wandering domains in iteration of meromorphic functions. We will not discuss such domains here, but an important property of this scenario is that the induced sequence $(f_n)_{n\in\mathbb{N}}$ of self-maps of $\D$ is entirely made of inner functions (see \cite[Section 3]{BEFRS19} and \cite[Lemma 2.2]{Fer21}). By \textit{inner function}, we mean a holomorphic function $f:\D\to\D$ such that the radial limits
\[ f(e^{i\theta}) := \lim_{r\nearrow 1} f(re^{i\theta}) \]
exist and satisfy $|f(e^{i\theta})| = 1$ for Lebesgue-a.e. $\theta\in [0, 2\pi)$.

A composition of two inner functions yields again an inner function -- see, for example, \cite[Theorem 2]{Ryf66}. Our primary aim here is to extend this result to limits of forward iteration of inner functions. Notice that knowing that locally uniform convergence takes place is not enough; in fact, as shown by Carath\'eodory \cite{Car54}, \textit{any} holomorphic function $f:\D\to\D$ can be locally uniformly approximated by finite Blaschke products (which are, of course, inner functions). Nevertheless, the fact that forward iteration is primarily a composition allows us to prove the following.
\begin{theorem} \label{thm:inner}
Let $f_n:\D\to\D$, $n\in\mathbb{N}$, be inner functions such that $f_n(0) = 0$ for all $n\in\mathbb{N}$, and assume that $F_n := f_n\circ\cdots\circ f_1$ converges locally uniformly to some $F:\D\to\D$. Then, $F$ is either constant or an inner function, and the latter happens if and only if
\[ \sum_{n\geq 1} (1 - |f_n'(0)|) < +\infty. \]
\end{theorem}

After proving Theorem \ref{thm:inner} in Section \ref{sec:inner}, we turn to the special case of Blaschke products -- see \cite{Dur70} and \cite{GMR16} for more properties of such functions. The composition of two \textit{finite} Blaschke products is once again a Blaschke product, but that is \textit{not} true in general for infinite Blaschke products -- it is true for a specific class of Blaschke products called \textit{indestructible} Blaschke products; see \cite{McL72} and \cite{KR13}. In view of that, the following result is of some interest.
\begin{theorem} \label{thm:blaschke}
Let $b_n:\D\to\D$ be finite Blaschke products such that $b_n(0) = 0$ for all $n\in\mathbb{N}$, and assume that $B_n := b_n\circ\cdots\circ b_1$ converges locally uniformly to some $B:\D\to\D$. Then, $B$ is either constant or a Blaschke product.
\end{theorem}

After that is established, we move on to consider the boundary extensions of forward iterations of Blaschke products. As mentioned previously, an inner function $f$ admits radial limits almost everywhere on $\mathbb{T} = \partial\D$. This means that one can define the \textit{boundary extension} of $f$ to be
\[ \hat{f}(e^{i\theta}) := \lim_{r\nearrow 1} f(re^{i\theta}); \]
the hat notation here serves to emphasise that we should think of $\hat{f}$ as a function $\hat{f}:\mathbb{T}\to\Cx$ belonging to $L^p(\mathbb{T})$ for $p\in[1, +\infty]$; see \cite{Dur70} and \cite[Chapter 17]{Rud87} for more on $H^p$ spaces and their relation to $L^p(\mathbb{T})$. We drop the $\mathbb{T}$ from this notation, since we will not be talking about any other $L^p$ spaces. Of course, since $f$ is an inner function, one has $|\hat{f}(e^{i\theta})| = 1$ almost everywhere, so that in fact one can think of $\hat{f}$ as a measurable self-map of $\mathbb{T}$, defined up to a set of measure zero. In our setting, this means that the forward iteration of the inner functions $(f_n)_{n\in\mathbb{N}}$ induces the forward iteration of $(\hat{f}_n)_{n\in\mathbb{N}}$ on $\mathbb{T}$. A natural question then is: if the functions $F_n := f_n\circ\cdots\circ f_1$ converge locally uniformly to $F$, do the boundary extensions $\hat{F}_n := \hat{f}_n\circ\cdots \hat{f}_1$ converge (in any sense) to $\hat{F}$? We give sufficient conditions for a positive answer:
\begin{theorem} \label{thm:boundary}
Let $b_n:\D\to\D$ be finite Blaschke products such that $b_n(0) = 0$, $b_n'(0) > 0$ and $d_n := \deg b_n - 1 < M < +\infty$ for all $n\in\mathbb{N}$. Assume that
\begin{equation} \label{eq:FC}
    \sum_{n\geq 1} (1 - b_n'(0))\log\frac{1}{1 - b_n'(0)} < +\infty.
\end{equation}
Then, the functions $B_n := b_n\circ\cdots\circ b_1$ converge locally uniformly in $\D$ to a Blaschke product~$B$, and the boundary extensions of $\hat{B}_n$ converge to $\hat{B}$ in $L^1$.
\end{theorem}
\begin{remark}
The condition $b_n'(0) > 0$ is stated only to guarantee that the sequence $(B_n)_{n\in\mathbb{N}}$ converge to a \textit{unique} limit. We could do without it, as long as we assume that the sequence is converging to a unique limit.
\end{remark}

The summability condition (\ref{eq:FC}) was first introduced by Frostman \cite{Fro42} to study the boundary convergence of infinite Blaschke products, and hence we will call it the \textit{Frostman condition}. Rybkin \cite{Ryb91} too investigated this and other ``Frostman conditions''; his results are closely related to ours, but more related to $L^p$ convergence of boundary extensions of infinite Blaschke products. Linden \cite{Lin76} showed that one can construct an infinite Blaschke product for which this condition fails and that diverges everywhere on $\mathbb{T}$. In the same spirit, we construct a sequence $(b_n)_{n\in\mathbb{N}}$ for which the Frostman condition fails and the boundary extensions of its forward iteration diverges everywhere on $\mathbb{T}$.
\begin{theorem} \label{thm:ex}
Let $(r_n)_{n\in\mathbb{N}}$ be a sequence in $(0, 1)$ such that
\[ \text{$\sum_{n\geq 1} (1 - r_n) < +\infty$\quad but\quad $\sum_{n\geq 1} (1 - r_n)\log\frac{1}{1 - r_n} = +\infty$.} \]
Let
\[ \theta_n := \sum_{i=1}^n (1 - r_i)\log\frac{1}{1 - r_i}, \]
let $z_n := r_ne^{i\theta_n}$, and define
\[ b_n(z) = z\cdot\frac{|z_n|}{z_n}\frac{z_n - z}{1 - \overline{z_n}z}. \]
Then, $B_n := b_n\circ\cdots\circ b_1$ converges locally uniformly in $\D$ to an infinite Blaschke product $B$, but the boundary extensions $\hat{B}_n$ do not converge to $\hat{B}$ at any point of $\mathbb{T}$.
\end{theorem}
The proof of Theorem \ref{thm:ex} is very similar to that of \cite[Theorem]{Lin76}. As such, we will provide only a brief outline at the end of Subsection \ref{ssec:boundary}.

\noindent\textsc{Acknowledgements.} I would like to thank my supervisors, Phil Rippon and Gwyneth Stallard, for their encouragement, comments, and suggestions -- especially regarding Theorem~\ref{thm:inner}.

\section{Forward iteration of inner functions} \label{sec:inner}
In this section, we use estimates from \cite[Corollary 2.4 and Theorem 2.5]{BEFRS19} to prove Theorem~\ref{thm:inner}. More specifically, we have:
\begin{lemma} \label{lem:estimates}
Let $g:\D\to\D$ be a holomorphic function with $g(0) = 0$ and $|g'(0)| = \lambda$. Then, for $z\in\D$ such that $|z| \leq \lambda$,
\[ |g(z)| \geq |z|\left(1 - \mu\cdot\frac{1 + |z|}{1 - |z|}\right), \]
where $\mu = 1 - \lambda$.
\end{lemma}
\begin{proof}[Proof of Theorem \ref{thm:inner}]
Let us assume that $F = \lim_{n\to +\infty} F_n$ is non-constant. It is clear from Weiestrass's theorem that $F$ is holomorphic; we must prove the stronger fact that it is an inner function. By Theorem \ref{thmout:class}, we have
\[ \sum_{n\geq 1} (1 - |f_n'(0)|) < +\infty; \]
in other words, $\prod_{n\geq 1} |f_n'(0)|$ converges to a positive number, and so given $\epsilon > 0$ there exists $N = N(\epsilon)$ such that
\[ \nu := \prod_{n=N+1}^{+\infty} |f_n'(0)| > 1 - \epsilon. \]
Now, for $n > N$, let
\[ H_n(z) := f_n\circ\cdots\circ f_{N+1}; \]
by Montel's theorem, there exists a subsequence $(H_{n_k})_{k\in\mathbb{N}}$ converging locally uniformly to some $H:\D\to\D$, and by Schwarz's lemma we in fact have $|H_n|\to |H|$. It follows that $|H'(0)| = \nu$, so that by Lemma \ref{lem:estimates} we have
\[ |H(z)| \geq |z|\left(1 - (1 - \nu)\frac{1 + |z|}{1 - |z|}\right) \geq |z|\left(1 - \epsilon\cdot\frac{1 + |z|}{1 - |z|}\right) \]
for $|z| < \nu$. Applying this to $|z| = 1 - \epsilon^{1/2} < 1 - \epsilon < \nu$, we get
\[ |H(z)| \geq (1 - \epsilon^{1/2})\left(1 - \epsilon\cdot\frac{2 - \epsilon^{1/2}}{\epsilon^{1/2}}\right) = (1 - \epsilon^{1/2})(1 - 2\epsilon^{1/2} + \epsilon), \]
and since $\epsilon > 0$ this gives
\begin{equation} \label{eq:H}
    |H(z)| \geq 1 - 3\epsilon^{1/2}.
\end{equation}
Define now the set $\Gamma_\epsilon := \{z : |F_N(z)| = 1 - \epsilon^{1/2}\}$; then, by (\ref{eq:H}),
\begin{equation} \label{eq:HGN}
    \text{$|F(z)| = |H\circ F_N(z)| \geq 1 - 3\epsilon^{1/2}$ for $z\in \Gamma_\epsilon$.}
\end{equation}
Furthermore, by Schwarz's lemma,
\begin{equation} \label{eq:Ge}
    \Gamma_\epsilon \subset \{z : 1 - \epsilon^{1/2} < |z| < 1\}.
\end{equation}
It is not true in general that for any inner function $f:\D\to\D$ fixing the origin the set $\{z : |f(z)| = r\}$ surrounds the circle $\{z : |z| = r\}$, even for $r$ very close to one; think, for instance, of the inner function $z\mapsto z\cdot\exp\left((z - 1)/(z + 1)\right)$. Nevertheless, we can show that $\Theta_\epsilon := \{\arg z : z\in\Gamma_\epsilon\}$ always has full measure in $[0, 2\pi)$, and that $\epsilon' < \epsilon$ implies $\Theta_\epsilon'\subset \Theta_\epsilon$.

Indeed, since $F_N$ is a composition of finitely many inner functions, it is itself an inner function, so that the set
\[ I_N := \{\theta\in[0, 2\pi) : \lim_{r\nearrow 1} |F_N(re^{i\theta})| = 1\} \]
has full Lebesgue measure. Applying the intermediate value theorem to the function $r\mapsto |F_N(re^{i\theta})|$, we see that for $\theta\in I_N$ there is always some $r^*\in (0, 1)$ such that $|F_N(r^*e^{i\theta})| = 1 - \epsilon^{1/2}$, and hence $I_N\subset \Theta_\epsilon$. Since $I_N$ has full measure in $[0, 2\pi)$, so does $\Theta_\epsilon$, as claimed. Now, if $\epsilon' < \epsilon$, we can assume that $N' = N(\epsilon') \geq N = N(\epsilon)$, so that by Schwarz's lemma $|F_{N'}(z)| \leq |F_N(z)|$ for all $z\in\mathbb{D}$. For $z' = r'e^{i\theta'}\in\Gamma_{\epsilon'}$, this means that $|F_N(z')| \geq |F_{N'}(z')| = 1 - \epsilon^{1/2}$, meaning that we can once again apply the intermediate value theorem and conclude that $\theta'\in \Theta_\epsilon$, and therefore $\Theta_{\epsilon'}\subset \Theta_\epsilon$.

Taking now a sequence $\epsilon_k \searrow 0$, we see that the set
\[ \Theta := \bigcap_{k\geq 1} \Theta_{\epsilon_k} \]
has full Lebesgue measure in $[0, 2\pi)$, and for $\theta\in\Theta$ we have by (\ref{eq:HGN}) and (\ref{eq:Ge}) that
\[ \limsup_{r\nearrow 1} |F(re^{i\theta})| = 1. \]
By Fatou's theorem, $F$ has well-defined radial limits in a full-measure set $E\subset [0, 2\pi)$, and the set $E\cap\Theta$ has full Lebesgue measure. For $\theta\in E\cap\Theta$, we must therefore have $|F(e^{i\theta})| = 1$; it follows that $F$ is an inner function.
\end{proof}

\section{The special case of Blaschke products}
Here, we prove Theorems \ref{thm:blaschke} and \ref{thm:boundary}.
\subsection{Convergence to a Blaschke product}
To prove Theorem \ref{thm:blaschke}, we must pay attention to the zeros of $B$ and $B_n$. Denoting by $Z(f)$ the set of zeros of the holomorphic function $f$, we have:
\begin{lemma} \label{lem:zeros}
Under the conditions of Theorem \ref{thm:blaschke} with $B$ non-constant, we have
\begin{equation*}
    \text{$Z(B_n)\subset Z(B_{n+1})$, $n\in\mathbb{N}$, and $Z(B) = \bigcup_{n\geq 1} Z(B_n)$.}
\end{equation*}
In particular,
\begin{equation} \label{eq:BC}
    \sum_{z\in\bigcup_n Z(B_n)} (1 - |z|) < +\infty.
\end{equation}
\end{lemma}
\begin{proof}
Since $b_{n+1}(0) = 0$ and $B_{n+1} = b_{n+1}\circ B_n$, it follows that $Z(B_n)\subset Z(B_{n+1})$, whence
\begin{equation} \label{eq:zeros}
    Z(B)\supset \bigcup_{n\geq 1} Z(B_n).
\end{equation}

To prove the reverse inclusion, take a zero $z^*$ of $B$. Since $B$ is holomorphic and non-constant, there exists a small disc $D = \{z : |z - z^*| < r\}$ such that $\overline{D}\subset\D$ and $\overline{D}\cap Z(B) = \{z^*\}$. By Hurwitz's theorem, $B_n$ also has a zero $z_n$ in $D$ for all large $n$; by (\ref{eq:zeros}), we must have $z_n = z^*$, proving the reverse inclusion.

Finally, (\ref{eq:BC}) follows from the fact that $B$ is a bounded non-constant function, and thus $Z(B)$ satisfies the Blaschke condition by \cite[Theorem 2.3]{Dur70}.
\end{proof}

We are ready to prove Theorem \ref{thm:blaschke}.
\begin{proof}[Proof of Theorem \ref{thm:blaschke}]
The first thing to notice is that $B_n$, being a finite composition of finite Blaschke products, is a finite Blaschke product. The degree $D_n$ of $B_n$ is given by
\[ D_n = \prod_{i=1}^n \deg b_i, \]
and it follows that we can decompose $B_n$ as
\[ B_n(z) = \alpha_n\cdot z\cdot \prod_{i=1}^{D_n} \frac{|z_i|}{z}\frac{z_i - z}{1 - \overline{z_i}z}, \]
where $|\alpha_n| = 1$ and $z_i$ runs over the zeros of $B_n$ excluding the origin. Since $Z(B_{n+1})\supset Z(B_n)$ by Lemma \ref{lem:zeros}, we have
\[ B_{n+1}(z) = \frac{\alpha_{n+1}}{\alpha_n}\left(\prod_{i=1}^{D_{n+1} - D_n} \frac{|z_i|}{z_i}\frac{z_i - z}{1 - \overline{z_i}z}\right)B_n(z), \]
where $z_i$ runs over $Z(B_{n+1})\setminus Z(B_n)$. In other words, the sequence $(B_n)_{n\in\mathbb{N}}$ represents a multiplicative sequence of Blaschke factors, which converges locally uniformly to an infinite Blaschke product by (\ref{eq:BC}) and \cite[Theorem 2.4]{Dur70}. This infinite Blaschke product is $B$, and we are done.
\end{proof}
\begin{remark}
It is clear from the proof that Theorem \ref{thm:blaschke} can be extended to forward iteration of indestructible Blaschke products. A natural question is whether the limit is itself indestructible.
\end{remark}

\subsection{Convergence of boundary extensions} \label{ssec:boundary}
First, we show that if a forward iterated sequence of inner functions converges to a non-zero limit, then there is always \textit{some} ``good behaviour'' of the boundary extensions.
\begin{lemma} \label{lem:id}
Let $f_n:\D\to\D$ be inner functions such that $f_n(0) = 0$ and $f_n'(0) > 0$ for all $n\in\mathbb{N}$. If $F_n := f_n\circ\cdots\circ f_1$ converges locally uniformly to a non-constant function $F:\D\to\D$, then
\[ \frac{\hat{f}_n(e^{i\theta})}{e^{i\theta}} \to 1 \]
for Lebesgue-a.e. $\theta\in[0, 2\pi)$.
\end{lemma}
\begin{proof}
For $n\in\mathbb{N}$, define the functions
\[ \psi_n(z) := \begin{cases} \frac{f_n(z)}{z}, & z\in\D\setminus\{0\}, \\ f_n'(0), & z = 0. \end{cases} \]
By the removable singularity theorem, these functions are holomorphic in $\D$, and by Schwarz's lemma they actually satisfy $\psi_n:\D\to\D$. Furthermore, since $F_n\to F\not\equiv 0$, we also have
\[ \sum_{n\geq 1} (1 - \psi_n(0)) = \sum_{n\geq 1} (1 - f_n'(0)) < +\infty \]
by Theorem \ref{thmout:class}. It then follows from \cite[Theorem B]{BEFRS22} that
\[ \text{$|\psi_n(e^{i\theta}) - \psi_n(z)|\to 0$ for $z\in\D$ as $n\to+\infty$} \]
for Lebesgue-a.e. $\theta\in[0, 2\pi)$. Since $\psi_n(0)\to 1$ as $n\to +\infty$, it now follows from \cite[Theorem A]{BEFRS22} that $\psi_n(z)\to 1$ for all $z\in\mathbb{D}$, and by the triangle inequality $\psi_n(e^{i\theta})\to 1$ for Lebesgue-a.e.~$\theta$.
\end{proof}

Lemma \ref{lem:id} shows that, if $f_n$ are inner functions fixing the origin such that $f_n'(0) > 0$ and $(F_n)$ is a semi-contracting sequence (see \cite[Theorem 7.2]{BEFRS22} for this terminology), then $\hat{f}_n$ converges to the identity Lebesgue almost everywhere, which \textit{might} imply that some kind of convergence happens for $\hat{F}_n$. Of course, our point here is that things are not so simple; nevertheless, we offer a sufficient condition for that to be the case, thus proving Theorem \ref{thm:boundary}.
\begin{proof}[Proof of Theorem \ref{thm:boundary}]
Let
\[ \psi_n(e^{i\theta}) := \frac{\hat{B}_n(e^{i\theta})}{e^{i\theta}}; \]
we will show that $(\psi_n)_{n\in\mathbb{N}}$ is a Cauchy sequence in $L^1$. As $H^1$ is a closed subspace of $L^1$, which is a Banach space, this in turn implies that $(\hat{B}_n)_{n\in\mathbb{N}}$ converges in $L^1$, and our theorem will follow from the fact that convergence of the boundary extensions in $L^1$ implies locally uniform convergence in $\D$ (see \cite[Theorem 3.3]{Dur70} and the following corollary, or \cite[Remark 17.8(c) and Theorem 17.11]{Rud87}).

To that end, let $\dist(e^{i\theta}, e^{i\phi})$ denote distance in $\mathbb{T}$, so that
\[ \dist(\psi_n(e^{i\theta}), \psi_{n+m}(e^{i\theta})) = \left|\arg\frac{\hat{B}_{n+m}(e^{i\theta})}{\hat{B}_n(e^{i\theta})}\right| \]
for all $m, n\in\mathbb{N}$, where $\arg z\in [-\pi, \pi)$. We rewrite the right-hand quotient as
\[ \frac{\hat{B}_{n+m}(e^{i\theta})}{\hat{B}_n(e^{i\theta})} = \prod_{i=1}^m \frac{\hat{B}_{n+i}(e^{i\theta})}{\hat{B}_{n+i-1}(e^{i\theta})} = \prod_{i=1}^m \frac{\hat{b}_{n+i}\left(\hat{B}_{n+i-1}(e^{i\theta})\right)}{\hat{B}_{n+i-1}(e^{i\theta})}, \]
so that
\begin{equation} \label{eq:ineqC}
    \dist(\psi_{n+m}(e^{i\theta}), \psi_n(e^{i\theta})) = \left|\arg\prod_{i=1}^m\frac{\hat{b}_{n+i}\left(\hat{B}_{n+i-1}(e^{i\theta})\right)}{\hat{B}_{n+i-1}(e^{i\theta})}\right| \leq \sum_{i=1}^m \left|\arg\frac{\hat{b}_{n+i}\left(\hat{B}_{n+i-1}(e^{i\theta})\right)}{\hat{B}_{n+i-1}(e^{i\theta})}\right|.
\end{equation}
In order to relate this to $L^1$, we note that,
\[ |e^{i\theta} - e^{i\phi}| \leq \dist(e^{i\theta}, e^{i\phi}) \]
for all $\theta, \phi\in[0, 2\pi)$, meaning that (\ref{eq:ineqC}) implies
\[ \|\psi_{n+m} - \psi_n\|_1 = \frac{1}{2\pi}\int_0^{2\pi} |\psi_{n+m}(e^{i\theta}) - \psi_n(e^{i\theta})|\,d\theta \leq \sum_{i=1}^m\frac{1}{2\pi}\int_0^{2\pi}\left|\arg\frac{\hat{b}_{n+i}\left(\hat{B}_{n+i-1}(e^{i\theta})\right)}{\hat{B}_{n+i-1}(e^{i\theta})}\right|\,d\theta. \]
Our next step is to recall that, as shown by Doering and Ma\~n\'e \cite[Corollary 1.5]{DM91}, inner functions fixing the origin preserve the (normalised) Lebesgue measure on $\mathbb{T}$ in the sense that $\mu\circ \hat{f}^{-1} = \mu$ (for $\mu$ being the normalised Lebesgue measure and $\hat{f}$ the boundary extension of an inner function fixing the origin), and so the inequality above can be rewritten as
\begin{equation} \label{eq:argint}
    \|\psi_{n+m} - \psi_n\|_1 \leq \sum_{i=1}^m\frac{1}{2\pi}\int_0^{2\pi}\left|\arg\frac{\hat{b}_{n+i}\left(e^{i\theta}\right)}{e^{i\theta}}\right|\,d\theta.
\end{equation}
Since we assume the functions $b_n$ to be finite Blaschke products of degree $d_n + 1$, we can decompose them as
\[ b_n(z) = z\prod_{i=1}^{d_n}\frac{|z_i|}{z_i}\frac{z_i - z}{1 - \overline{z_i}z}, \]
where $z_i = z_i(n)$, $1\leq i\leq d_n$, are the non-zero zeros of $b_n$. This gives us the following identity (see \cite[Theorem 2.4]{GMR16} or \cite[Equation 2.1]{Ryb91}):
\[ \arg\frac{\hat{b}_n(e^{i\theta})}{e^{i\theta}} = -2\sum_{i=1}^{d_n} \arctan\left(\frac{1 - |z_i|}{(1 + |z_i|)\tan\left(\frac{\theta - \theta_i}{2}\right)}\right), \]
where $\theta_i = \arg z_i$ and $\theta\neq\theta_i$ for $1\leq i \leq d_n$, and $\arctan$ denotes the principal branch of the arctangent (and, in particular, has range $(-\pi/2, \pi/2)$). Inserting this identity into (\ref{eq:argint}) and applying the triangle inequality yields
\[ \|\psi_{n+m} - \psi_n\|_1 \leq \sum_{i=1}^m\sum_{j=1}^{d_{n+i}}\frac{1}{2\pi}\int_0^{2\pi}\left|2\arctan\left(\frac{1 - |z_j|}{(1 + |z_j|)\tan\left(\frac{\theta - \theta_j}{2}\right)}\right)\right|\,d\theta. \]
A direct application of the estimates given by Rybkin for the integral on the right-hand side \cite[Proof of Theorem 3]{Ryb91} gives
\[ \|\psi_{n+m} - \psi_n\|_1 \leq \sum_{i=1}^m \sum_{j=1}^{d_{n+i}} \left((2 + \log 2\pi^2)(1 - |z_j|) + 2(1 - |z_j|)\log\frac{1}{1 - |z_j|}\right). \]
Rearranging the sums and denoting the smallest non-zero zero of $b_n$ by $z_n^*$, we obtain
\[ \|\psi_{n+m} - \psi_n\|_1 \leq (2 + \log 2\pi^2)M\sum_{i=1}^m (1 - |z_n^*|) + 2M\sum_{i=1}^m (1 - |z_n^*|)\log\frac{1}{1 - |z_n^*|}, \]
and we see that we can make the sum on the right-hand side arbitrarily small by taking a large enough $n$, since $|z_n^*| \geq b_n'(0)$ by \cite[Corollary 2.4]{BEFRS19} and the sequence $(b_n'(0))_{n\in\mathbb{N}}$ satisfies the Frostman condition (\ref{eq:FC}) by hypothesis. It follows that $(\psi_n)_{n\in\mathbb{N}}$ is a Cauchy sequence in $L^1$, and we are done.
\end{proof}
\begin{remark}
As an alternative proof of Theorem \ref{thm:boundary}, we could start from the ``product decomposition'' used in the proof of Theorem \ref{thm:blaschke} and attempt to apply \cite[Theorem 3]{Ryb91}. However, unlike in the proof of Theorem \ref{thm:blaschke}, it is not straightforward to show that the zeros of $B_n$ collectively satisfy the Frostman condition.
\end{remark}

Finally, we outline the proof of Theorem \ref{thm:ex}. Most of the necessary tools were already introduced in the proofs of Theorem \ref{thm:blaschke} and \ref{thm:boundary}; the rest is due to Linden \cite{Lin76}.
\begin{proof}[Outline of the proof of Theorem \ref{thm:ex}]
The compositions $B_n := b_n\circ\cdots\circ b_1$ are all finite Blaschke products, and as such can be rewritten as a product
\[ B_n(z) = \alpha_n\cdot z\cdot \prod_{i=1}^{D_n} \frac{|z_i|}{z_i}\frac{z_i - z}{1 - \overline{z_i}z} \]
in the spirit of the proof of Theorem \ref{thm:blaschke}. Therefore, recalling the proof of Theorem \ref{thm:boundary}, its boundary extension satisfies
\[ \arg \hat{B}_n(e^{i\theta}) = \theta - 2\sum_{i=1}^{D_n} \arctan\left(\frac{1 - |z_i|}{(1 + |z_i|)\tan\left(\frac{\theta - \theta_i}{2}\right)}\right); \]
due to our careful choices of $z_i$, the same arguments in the proof of Theorem \cite[Theorem]{Lin76} show that $(\hat{B}_n(e^{i\theta}))_{n\in\mathbb{N}}$ is not a Cauchy sequence at any point of $\mathbb{T}$.
\end{proof}

\end{document}